\documentclass[leqno, 11pt]{article}

\usepackage{amsmath, amssymb}
\usepackage{amsthm} 
\usepackage{braket}

\makeatletter 
\def\section{\@startsection {section}{1}{\z@}{-3.5ex plus -1ex minus -.2ex}{2.3 ex plus .2ex}{\centering\large\bf}}
\makeatother

\theoremstyle{plain} 
\newtheorem{theorem}{\noindent\bf Theorem}[section] 
\newtheorem{lemma}[theorem]{\noindent\bf Lemma}
\newtheorem{corollary}[theorem]{\noindent\bf Corollary}
\newtheorem{proposition}[theorem]{\noindent\bf Proposition}

\theoremstyle{definition} 
\newtheorem{definition}[theorem]{\noindent\bf Definition}
\newtheorem{remark}[theorem]{\noindent\bf Remark}
\newtheorem{example}[theorem]{\noindent\bf Example}

\newcommand{\Vol}[0]{\operatorname{Vol}}

 \makeatletter
    
    \@addtoreset{equation}{section}
  \makeatother

\makeatletter

\def\address#1#2{\begingroup
\noindent\parbox[t]{7.8cm}{
\small{\scshape\ignorespaces#1}\par\vskip1ex
\noindent\small{\itshape E-mail address}
\/: #2\par\vskip4ex}\hfill
\endgroup}
\makeatother

\title{Logarithmic Chow semistability of polarized toric manifolds}

\author{Satoshi Nakamura} 

\date{}

\begin{document}

\maketitle




\begin{abstract}
The logarithmic Chow semistability is a notion of Geometric Invariant Theory for the pair consists of varieties and its divisors. 
In this paper we introduce a obstruction of semistability for polarized toric manifolds and its toric divisors.
As its application, we show the implication from the asymptotic log Chow semistability to the log K-semistability by combinatorial arguments.
Furthermore we give a non-semistable example which has a conical K\"{a}hler Einstein metric.
\end{abstract}

\section{Introduction}


Let $X$ be a compact complex manifold and $L\to X$ an ample line bundle.
We call this pair $(X,L)$ a polarized manifold.
The well-known Donaldson-Tian-Yau conjecture claims that $X$ has constant scalar curvature K\"{a}hler metrics in $c_1(L)$ if and only if $(X,L)$ is stable in the sense of Geometric Invariant Theory (GIT for short).
This conjecture has the logarithmic generalization which claims that
$X$ has constant scalar curvature conical K\"{a}hler metric in  $c_1(L)$ with cone angle $2\pi\beta$ along a divisor $D$ if and only if $(X,D,L,\beta)$ is stable in the sense of GIT.
The one of difficulty of this problem is that there are various notions of stability.
Therefore it is important to see the relation among various stabilities and the relation between each stability and the existence of such metrics.

In this paper, we mainly consider the logarithmic Chow semistability and logarithmic K-semistability for the pair consists of polarized toric manifolds and its toric divisors.

It is well-known that the one to one correspondence between $n$-dimensional integral Delzant polytopes and $n$-dimensional compact toric manifolds with $(\mathbb{C}^*)^n$-equivariant very ample line bundles.
Futhermore every facet of integral Delzant polytopes corresponds to a divisor, called toric divisor, which is invariant under $(\mathbb{C}^*)^n$-action.

The following is the main theorem of this paper. 
Set $E_P(i):=\#\{P\cap(\mathbb{Z}/i)^n\}$ and $T_{iP}$ be the standard complex torus in $SL(E_P(i))$.
(We refer to later sections for notations.)

\begin{theorem}\label{main}
Let $P\subset\mathbb{R}^n$ be a $n$-dim integral Delzant polytope, $F$ a facet of $P$ and $\beta \in (0,1]$ a cone angle.
For a positive integer $i$, the Chow form of $(X_P, D_F, L_P^i, \beta)$
is $T_{iP}$-semistable if and only if  
\[ 
E_P(i)  \biggl( 2i \int_P g d\nu + (1-\beta)\int_F g d\sigma \biggr)
\geq \biggl( 2i \mathrm{Vol}(P) + (1-\beta)\mathrm{Vol}(F) \biggr) \sum_{\mathbf{a} \in P \cap (\mathbb{Z}/i)^n} g(\mathbf{a})
\]
holds for all concave piecewise linear function $g :P \to \mathbb{R}$ in $\mathrm{PL}(P,i)$.
\end{theorem}

The organization of this paper is as follows.
In section 2, we review some fundamentals of GIT for later use.
In section 3, we define  logarithmic Chow semistability in the general setting.
In section 4, we consider  logarithmic Chow semistability for polarized toric manifolds and its toric divisors and prove Theorem \ref{main}.
In section 5, we compare with asymptotic logarithmic Chow semistability and logarithmic K-semistability as one of applications of Theorem \ref{main}.
In section 6, as the other applications of Theorem \ref{main}, we give an example of asymptotic log Chow unstable manifold which admits conical K\"{a}hler Einstein metrics.

{\bf Acknowlegdements.}
The author would like to thank Professor Shigetoshi Bando and Doctor Ryosuke Takahashi for several helpful comments and constant encouragement.
He also would like to thank Professor Naoto Yotsutani for careful reading the draft of this paper.

\section{GIT stability}

We recall some facts on GIT stability, see for example  \cite{Mum2}.
Let $V$ be a finite dimensional complex vector space with linear action of a reductive Lie group $G$. 
We say non zero vector $v \in V$ is $G$-semistable if the closure of the orbit $G\cdot v$ does not contain $0\in V$.
It is well-known the following criterion for semistability.

\begin{proposition}\label{Hilbert-Mumford}
$v\in V$ is $G$-semistable if and only if $v$ is $T$-semistable for all maximal torus $T \subset G$.
\end{proposition}

Therefore it is important to study the torus action.
Let $T$ be a algebraic torus $(\mathbb{C}^*)^N$ for some $N$ in $G$.
Then $V$ can be decomposed as 
\[
V = \sum_{\chi\in\chi(T)}\Set{v \in V | t\cdot v = \chi(t)v, \forall t \in T},
\]
where $\chi(T)$ is the character group of $T$. 
Note that $\chi(T)$ is isomorphic to $\mathbb{Z}^N$ since we can express each $\chi \in \chi(T)$ as a Laurent monomial
$\chi(t_1, \dots, t_N)=t_1^{a_1}\cdots t_N^{a_N}$ where $t_i \in \mathbb{C}^*$ and $a_i\in \mathbb{Z}$.

\begin{definition}
Let $v=\sum_{\chi\in\chi(T)}v_{\chi}$ be a non zero vector in $V$. 
The weight polytope $\mathrm{Wt}_T(v)$ of $v$ is the convex hull of $\Set{\chi \in \chi(T) | v_{\chi} \neq 0}$ 
in $\chi(T)\otimes_{\mathbb{Z}}\mathbb{R} \cong \mathbb{R}^N$.
\end{definition}

$T$-semistability can be visualize by the weight polytope.

\begin{proposition}
$v$ is $T$-semistable if and only if the weight polytope $\mathrm{Wt}_T(v)$ of $v$ contains the origin.
\end{proposition}

For laters arguments we consider the following situation.
Let $H = (\mathbb{C^*})^{N+1}$ acting on $V$ and $T$ be a subtorus
\begin{equation}
\label{subtorus}
T:= \Set{ \bigl(t_1, \dots , t_{N}, (t_1\cdots t_{N})^{-1}\bigr) | (t_1, \dots , t_{N})\in (\mathbb{C^*})^{N} } 
\cong (\mathbb{C}^*)^{N}. 
\end{equation}
Then the weight polytope $\mathrm{Wt}_T(v)$  is coincide with $\pi (\mathrm{Wt}_H(v))$ 
where $\pi : \mathbb{R}^{N+1} \to \mathbb{R}^{N}$ is the linear map defined as $(x_1, \dots , x_{N+1}) \mapsto (x_1-x_{N+1}, \dots, x_{N}-x_{N+1})$.
Therefore we see the following.

\begin{proposition}\label{weight polytope}
$v$ is $T$-semistable if and only if there exists $t\in \mathbb{R}$ such that $(t, \dots , t) \in \mathrm{Wt}_H(v)$.
\end{proposition}

\section{Log Chow semistability}

\subsection{General definition}

First we recall definition of the Chow form of projective varieties. 
See \cite{Mum1, Mum2} for more detail.

Let $V$ be a finite dimensional complex vector space 
and $X \subset \mathbb{P}(V)$ be a $n$-dimensional irreducible projective variety of degree $d$.
Then 
\[
Z_X :=\Set{ (H_0, \dots, H_n) \in \mathbb{P}(V^*)^{n+1} | (\cap_{i=0}^n H_i)\cap X \neq \emptyset}
\]
has codimension 1 in $\mathbb{P}(V^*)^{n+1}$ and multi-dergree $(d,\dots, d)$, 
so there exists the {\it Chow form} $R_X \in (\mathrm{Sym}^d V)^{\otimes (n+1)}$, unique up to scalar multiplication, 
which vanishes on $Z_X$.

Although $(\mathrm{Sym}^d V)^{\otimes (n+1)}$ has the natural $SL(V)$-action induced from the natural $SL(V)$-action on $\mathbb{P}(V)$, 
we introduce a twisted action $SL(V) \times (\mathrm{Sym}^d V)^{\otimes (n+1)} \to (\mathrm{Sym}^d V)^{\otimes (n+1)}$ in order to define the log Chow semistability.
For the fixed $\alpha \in \mathbb{R}$, 
we define as $(\exp u, v) \mapsto \exp(\alpha u)v$
where $u \in \mathrm{Lie}(SL(V))$ and the action of $\exp(\alpha u)$ for $v$ is the natural one.
Then we define $R_X^{(\alpha)}$ as the Chow form of $X$ with $\alpha$-twisted $SL(V)$-action.
Note that $R_X^{(1)}$ is the ordinary Chow form $R_X$ with the natural $SL(V)$-action.

Next we define the log Chow semistability (\cite{WX}). 
Let $X$ be as above and $D$ be a irreducible divisor of degree $d'$ on $X$ and $\beta \in (0,1]$.

\begin{definition}
The pair $(X,D,\beta)$ is {\it log Chow semistable} 
if $R_X^{(2n!)}\otimes R_D^{((1-\beta)(n+1)!)}$ is $SL(V)$-semistable in 
$(\mathrm{Sym}^d V)^{\otimes (n+1)} \otimes (\mathrm{Sym}^{d'} V)^{\otimes n}$.
The pair $(X,D,\beta)$ is {\it log Chow unstable} if it is not log Chow semistable.
\end{definition}



\begin{definition}
Let  $(X,L)$ be a polarized variety, $D$ be a irreducible divisor on $X$ and $\beta \in (0,1]$.  
For $k\gg 0$, the pair $(X,D,L^k,\beta)$ is {\it log Chow semistable}
if $(\varphi_k(X), \varphi_k(D), \beta)$ is log Chow semistable
where $\varphi_k : X \to \mathbb{P}(H^0(X,L^k)^*)$ is the Kodaira embedding.
Further $(X,D,L,\beta)$ is {\it asymptotically log Chow semistable} if $(X,D,L^k,\beta)$ is log Chow semistable for all $k\gg 0$.

\end{definition}

\begin{remark}
(1) If we take $\beta =1$ then the log Chow semistability reduced to the Chow semistaility of $X$ defined in \cite{Mum1, Mum2}.

(2) The ratio $2n!:(1-\beta)(n+1)!$ is chosen so that 
asymptotic log Chow semistability implies log  K-semistability as discussed in Proposition \ref{K-semistability}.

(3) The log Chow semistability can be extended to $\mathbb{R}$-divisor $D=\sum_i (1-\beta_i)D_i$
where $D_i$ are defferent irreducible divisors and $\beta_i \in (0,1]$
by considering the tensor product $R_X^{(2n!)}\otimes ( \otimes_i R_{D_i}^{((1-\beta_i)(n+1)!)})$ 
\end{remark}

\subsection{A necessary condition}

We introduce a necessary condition for log Chow semistability of $(X,D,\beta)$
where $X \subset \mathbb{C}P^N$ is a $n$-dimensional irreducible subvariety and $D \subset X$ is an irreducible divisor.

Let $H :=(\mathbb{C}^*)^{N+1} \subset GL(N+1, \mathbb{C})$ be the torus consists of invertible diagonal matrix.
Then the subtorus $T$ defined by \eqref{subtorus} is a maximal torus of $SL(N+1,\mathbb{C})$.
By Proposition \ref{Hilbert-Mumford}, $T$-semistability of $R_X^{(2n!)}\otimes R_D^{((1-\beta)(n+1)!)}$  is a necessary condition of log Chow semistability of $(X,D,\beta)$.
Proposition \ref{weight polytope} implies the following.

\begin{proposition}\label{T-semistability}
$R_X^{(2n!)}\otimes R_D^{((1-\beta)(n+1)!)}$ is $T$-semistable if and only if
there exists $t\in\mathbb{R}$ such that 
\[
(t,\dots , t) \in \mathrm{Wt}_H(R_X^{(2n!)}\otimes R_D^{((1-\beta)(n+1)!)}) 
\subset \mathrm{Aff}_{\mathbb{R}}\bigl(\mathrm{Wt}_H(R_X^{(2n!)}\otimes R_D^{((1-\beta)(n+1)!)})\bigr).
\]
Here $\mathrm{Aff}_{\mathbb{R}}\bigl(\mathrm{Wt}_H(R_X^{(2n!)}\otimes R_D^{((1-\beta)(n+1)!)})\bigr)$
is the affine hull of the weight polytope 
$\mathrm{Wt}_H(R_X^{(2n!)}\otimes R_D^{((1-\beta)(n+1)!)})$
in $\chi(H)\otimes_{\mathbb{Z}}\mathbb{R} \cong \mathbb{R}^{N+1}$.
\end{proposition}

\begin{lemma}\label{decomposition}
\[
\mathrm{Wt}_H(R_X^{(2n!)}\otimes R_D^{((1-\beta)(n+1)!)}) = 2n! \mathrm{Wt}_H(R_X) + (1-\beta)(n+1)!\mathrm{Wt}_H(R_D).
\]
Here for any polytopes $P,Q \subset \mathbb{R}^{N+1}$ and for any $a, b \in \mathbb{R}$, 
we define the polytope $a P + b Q$ as $\Set{a x + b y | x \in P, y \in Q}$ which is called Minkowski summand.
\end{lemma}

\begin{proof}
By \cite[Example 1.6]{KSZ}, 
\[
\mathrm{Wt}_H(R_X^{(2n!)}\otimes R_D^{((1-\beta)(n+1)!)}) = \mathrm{Wt}_H(R_X^{(2n!)}) + \mathrm{Wt}_H(R_D^{((1-\beta)(n+1)!)}).
\]
Further since $R_X^{(2n!)}$ has the $2n!$-twisted $SL(N+1,\mathbb{C})$-action, we can see 
\[\mathrm{Wt}_H(R_X^{(2n!)}) = 2n!\mathrm{Wt}_H(R_X)\] by definition of the weight polytope.
Similarly, $\mathrm{Wt}_H(R_D^{((1-\beta)(n+1)!)})=(1-\beta)(n+1)!\mathrm{Wt}_H(R_D)$.
\end{proof}
\section{Log Chow semistability for polarized toric mamifolds}
\subsection{On varieties defined by finite sets}

First we consider projective varieties defined as follows.
Let $A:=\{a_0, \dots , a_N \} \subset \mathbb{Z}^n$ be a finite subset.
Suppose $A$ affinely generates $\mathbb{Z}^n \subset \mathbb{R}^n$ over $\mathbb{Z}$.
Then 
\begin{equation}\label{variety}
X_A := \text{closure of } \Set{ [x^{a_0} : \cdots : x^{a_{N}}]\in \mathbb{C}P^N | x \in (\mathbb{C}^*)^n }
\end{equation}
is an $n$-dimensional subvariety in $\mathbb{C}P^N$.

Again let $H:=(\mathbb{C}^*)^{N+1} \subset GL(N+1, \mathbb{C})$ be the torus of invertible diagonal matrix.

\begin{proposition}{\rm (\cite[Chapter 7, Proposition 1.11]{GKZ} and \cite{KSZ}.)}\label{GKZ}
When we set $P$ as the convex hull of $A$ in $\mathbb{Z}^n\otimes\mathbb{R}$, we have
\[
\mathrm{Aff}_{\mathbb{R}}(\mathrm{Wt}_H(R_{X_{A}})) = 
\Set{ \varphi : A \to \mathbb{R} | \sum_i \varphi(a_i) = (n+1)!\mathrm{Vol}(P), 
                                                   \sum_i \varphi(a_i)a_i = (n+1)!\int_P x d\nu }.
\]
Here $d\nu$ is the Euclidian volume form of $\mathbb{R}^n$ 
normalized as $\nu(\Delta_n)=1/n!$ for the $n$-dim standard simplex $\Delta_n$.
\end{proposition}
\begin{remark}
In above proposition, there is the canonical identification $\chi(H) \cong \{ \varphi : A \to \mathbb{Z}\}$.
\end{remark}

\subsection{Toric setting}

Let $P\subset \mathbb{R}^n$ be a $n$-dimensional integral Delzant polytope and let $(X_P, L_P)$ be the corresponding toric manifold and very ample line bundle.
Furthermore for any $i \in \mathbb{N}$, it is also well-known that the rescaled polytope $iP$ corresponds to $(X_P, L_P^i)$.
Then the image of $X_{P}$ of the Kodaira embedding by $L_P^i$ is given as
the form \eqref{variety} defined by the finite set $iP\cap\mathbb{Z}^n$.
Hereafter we identify $(X_{P},L_P^i)$ with its Kodaira embedding image in $\mathbb{P}(H^0(X_P, L_P^i)^*)$.

Now we investigate log Chow semistability of pairs consists of polarized toric manifolds and toric divisors.
Set $E_P(i) := \#\{iP\cap\mathbb{Z}^n\}$. 
By Proposition \ref{Hilbert-Mumford}, $T$-semistability of the Chow form is essential for any maximal torus $T\subset SL(E_P(i))$.
We take the following maximal torus as a special one:
\[
T_{iP}:=(\mathbb{C}^*)^{E_P(i)}\cap SL(E_P(i)).
\]
Here $(\mathbb{C}^*)^{E_P(i)} \subset GL(E_P(i))$ is the torus of invertible diagonal matrix and $T_{iP}$ is the subtorus given as \eqref{subtorus}.

Proposition \ref{GKZ} implies the following.
Set $\mathrm{Ch}(iP)$ as the weight polytope of the Chow form of $X_P \subset \mathbb{P}(H^0(X_P, L_P^i)^*)$ for $(\mathbb{C}^*)^{E_P(i)}$-action.

\begin{proposition}\label{chow form P} 
The affine hull of $\mathrm{Ch}(iP)$ in 
$\{\varphi : iP\cap\mathbb{Z}^n \to \mathbb{R}\} \cong \{\varphi : P\cap(\mathbb{Z}/i)^n \to \mathbb{R}\}$ is
\[
\Set{ \varphi : P\cap (\mathbb{Z}/i)^n \to \mathbb{R} | \sum_{\mathbf{a}\in P\cap (\mathbb{R}/i)^n} \varphi(\mathbf{a}) = (n+1)!\mathrm{Vol}(iP), 
                                                   \sum_{\mathbf{a}\in P\cap (\mathbb{R}/i)^n} \varphi(\mathbf{a})\mathbf{a} = (n+1)!\int_{iP} \mathbf{x} d\nu }.
\]
\end{proposition}

Let $F\subset P$ be a facet.
Then $F\subset P$ defines the toric divisor $D_{F} \subset X_{P}$ which is invariant under $(\mathbb{C^*})^n$-action on $X_P$.
If we write $iP\cap\mathbb{Z}^n=\{a_0, \dots , a_N\}$, then the image of $D_F$ of the Kodaira embedding by $L_P^i$ is given as
\[
\Set{[u_{a_0}:\cdots :u_{a_N}] \in X_P | u_{a_k}\equiv 0 \text{ iff } a_k \notin iF\cap\mathbb{Z}^n} \subset \mathbb{P}(H^0(X_P, L_P^i)^*).
\]

There is a canonical Euclidian volume form $d\sigma$ on $\partial P$ 
induced by $d\nu$ as follows:
On a facet $\{h_F=c_F\}\cap F$ of $F$, where $h_F$ is a primitive linear form, 
$dh_F \wedge d\sigma$ equals to the Euclidian volume form $d\nu$ on $\mathbb{R}^n$.

Therefore again Proposition \ref{GKZ} implies the following.
Set $\mathrm{Ch}(iF)$ as the weight polytope of the Chow form of $D_F \subset \mathbb{P}(H^0(X_P, L_P^i)^*)$ for $(\mathbb{C}^*)^{E_P(i)}$-action.

\begin{proposition}\label{chow form F}
The affine hull of $\mathrm{Ch}(iF)$ in 
$\{\varphi : P\cap(\mathbb{Z}/i)^n \to \mathbb{R}\}$ is
\[
\Set{ \varphi : P\cap (\mathbb{Z}/i)^n \to \mathbb{R} | \begin{array}{l} 
                                                                                     \displaystyle  \varphi \equiv 0 \text{ on } (P\setminus F)\cap(\mathbb{Z}/i)^n,\\
                                                                                      \displaystyle \sum_{\mathbf{a}\in P\cap (\mathbb{Z}/i)^n} \varphi(\mathbf{a}) = n!\mathrm{Vol}(iF),\\ 
                                                                                     \displaystyle \sum_{\mathbf{a}\in P\cap (\mathbb{Z}/i)^n} \varphi(\mathbf{a})\mathbf{a} = n!\int_{iF} \mathbf{x} d\sigma .
                                                                                       \end{array}
                                                                                      }.
\]
\end{proposition}

Summing up with Proposition \ref{T-semistability}, Lemma \ref{decomposition}, Proposition \ref{chow form P} and Proposition \ref{chow form F}, we have the following.

\begin{proposition}
The Chow form of $(X_P, D_F, L_P^i, \beta)$
is $T_{iP}$-semistable 
if and only if 
\begin{equation}
\label{iff}
\frac{n!(n+1)!}{E_P(i)}\Bigl(2\Vol(iP)+(1-\beta)\Vol(iF)\Bigr)d_{iP} 
\in 2n \mathrm{Ch}(iP)+ (1-\beta)(n+1)\mathrm{Ch}(iF),
\end{equation}
where $d_{iP}(\mathbf{a}) = 1$ for all $\mathbf{a}\in P\cap(\mathbb{Z}/i)^n$. 
\end{proposition}

To state the main theorem, 
we introduce the notion of the set consists on concave piecewise linear functions on $P$ induced from every $\varphi : P\cap (\mathbb{Z}/i)^n \to \mathbb{R}$.
For a  fixed such $\varphi$, when we set
\[
G_{\varphi} = \text{the convex hull of } \bigcap_{\mathbf{a}\in P \cap (\mathbb{Z}/i)^n} \Set{(\mathbf{a}, t) | t\leq \varphi(a)} \subset \mathbb{R}^n \times \mathbb{R},
\]
and define a piecewise linear function $g_{\varphi} : P \to \mathbb{R}$ by
\[
g_{\varphi}(\mathbf{x}) = \max\Set{t | (\mathbf{x},t) \in G_{\varphi}}.
\]

Here we set 
\[
\mathrm{PL}(P,i) = \Set{g_{\varphi} | \varphi : P\cap(\mathbb{Z}/i)^n \to \mathbb{R}}.
\]
\begin{remark}\label{Q}
For each $\varphi : P\cap(\mathbb{Z}/i)^n \to \mathbb{Q}$, the same construction above defines the $g_{\varphi} : P \to \mathbb{Q}$.
We also set
\[
\mathrm{PL}_{\mathbb{Q}}(P,i) = \Set{g_{\varphi} | \varphi : P\cap(\mathbb{Z}/i)^n \to \mathbb{Q}}.
\]
\end{remark}

When we further set 
\[\langle \varphi, \psi \rangle = \sum_{\mathbf{a}\in P\cap(\mathbb{Z}/i)^n} \varphi(\mathbf{a})\psi(\mathbf{a})
\]
 for every $\varphi, \psi : P\cap (\mathbb{Z}/i)^n \to \mathbb{R}$ as a scalar product,
then $g_{\varphi}$ has the following properties.

\begin{lemma}{\rm (\cite[Chapter 7, Lemma 1.9]{GKZ})}
\label{g}
For any $\varphi : P\cap(\mathbb{Z}/i)^n \to \mathbb{R}$,
\begin{enumerate}
\item the induced function $g_{\varphi}$ is concave.
\item we have the equality 
         \[
         \max\Set{\langle \varphi, \psi \rangle | \psi \in \mathrm{Ch}(iP)} = i^n(n+1)!\int_P g_{\varphi} d\nu.
         \]
         \end{enumerate}
\end{lemma}

Now we are ready to prove Theorem \ref{main}.



\begin{proof}
Since $2n! \mathrm{Ch}(iP)+ (1-\beta)(n+1)!\mathrm{Ch}(iF)$ is convex,
we can see that the condition \eqref{iff} holds if and only if 
\begin{equation*}
\begin{split}
&\max\Set{\langle \varphi, \psi \rangle | \psi \in 2n! \mathrm{Ch}(iP)+ (1-\beta)(n+1)!\mathrm{Ch}(iF)}\\
&\geq 
\frac{n!(n+1)!}{E_P(i)}\Bigl(2\Vol(iP)+(1-\beta)\Vol(iF)\Bigr) \langle \varphi , d_{iP} \rangle
\end{split} 
\end{equation*}
for all $\varphi : P\cap(\mathbb{Z}/i)^n \to \mathbb{R}$.
By the definition of Minkovski sum and Lemma \ref{g}, we have
\begin{equation*}
\begin{split}
&\max\Set{\langle \varphi, \psi \rangle | \psi \in 2n! \mathrm{Ch}(iP)+ (1-\beta)(n+1)!\mathrm{Ch}(iF)}\\
&= 2n! \max\Set{\langle \varphi, \psi \rangle | \psi \in \mathrm{Ch}(iP)}
       +(1-\beta)(n+1)! \max\Set{\langle \varphi, \psi \rangle | \psi \in \mathrm{Ch}(iF)}\\
&= i^{n-1}n!(n+1)!\Bigl(2i\int_P g_{\varphi}d\nu + (1-\beta)\int_F g_{\varphi}d\sigma \Bigr).
\end{split}
\end{equation*}
On the other hand, by the definition of $g_{\varphi}$, we have
\begin{equation*}
\begin{split}
&\frac{n!(n+1)!}{E_P(i)}\Bigl(2\Vol(iP)+(1-\beta)\Vol(iF)\Bigr) \langle \varphi , d_{iP} \rangle \\
&=\frac{i^{n-1}n!(n+1)!}{E_P(i)}\Bigl(2i\Vol(P)+(1-\beta)\Vol(F)\Bigr) \sum_{\mathbf{a}\in P\cap(\mathbb{Z}/i)^n} g_{\varphi}(\mathbf{a}).
\end{split}
\end{equation*}
Therefore we have the desired result.
\end{proof}

Since the $T_{iP}$-semistablity is a necessary condition for the asymptotic log Chow semistability, 
applying the above theorem to linear functions, we have the following.

\begin{corollary}\label{cor}
If $(X_P, D_F, L_P, \beta)$ is asymptotically log Chow semistable then
\begin{equation}\label{obstruction}
E_P(i)  \biggl( 2i \int_P \mathbf{x} d\nu + (1-\beta)\int_F \mathbf{x} d\sigma \biggr)
= \biggl( 2i \mathrm{Vol}(P) + (1-\beta)\mathrm{Vol}(F) \biggr) \sum_{\mathbf{a} \in P \cap (\mathbb{Z}/i)^n} \mathbf{a}
\end{equation}
holds for all $i\gg 0$.
\end{corollary}

\begin{remark}\label{extension}
(1) The reader should bear in mind that $E_P(i)$ (resp. $\sum_{a\in P\cap (Z/i)^n}a$) is a polynomial in $i$
(resp. $\mathbb{R}^n$-valued polynomial). Hence the identity theorem shows that \eqref{obstruction} must hold for every integer $i$
(cf. \cite[Theorem 1.4]{Ono1}).


(2) We can extend Theorem \ref{main} and Corollary \ref{cor} for any formal sum of different toric divisors $\sum_t (1-\beta_t)D_{F_t}$.
For instance, in the case of Corollary \ref{cor}, the equality (\ref{obstruction}) is replaced by
\[
E_P(i)  \biggl( 2i \int_P \mathbf{x} d\nu + \sum_t(1-\beta_t)\int_{F_t} \mathbf{x} d\sigma \biggr)
= \biggl( 2i \mathrm{Vol}(P) + \sum_t (1-\beta_t)\mathrm{Vol}(F_t) \biggr) \sum_{\mathbf{a} \in P \cap (\mathbb{Z}/i)^n} \mathbf{a}.
\]
\end{remark}

\begin{remark}
For $\beta=1$, Theorem \ref{main} and Corollary \ref{cor} are Ono's results in \cite{Ono1, Ono2}.
\end{remark}
\section{Relation to the log K-semistability}
As an application of Theorem \ref{main} we show that 
the asymptotic log Chow semistability for polarized toric manifolds and its toric divisors  
implies log K-semistability for toric degenerations.

Let us first recall the definition of the test configuration, the log Futaki invariant and the log K-semistability (See for example \cite{L} for more detail).
Let $(X,L)$ be a polarized manifold. 
\begin{definition}
A {\it test configuration} $(\mathcal{X},  \mathcal{L})$ for $(X,L)$ consists of
a $\mathbb{C}^*$-equivariant flat family $\mathcal{X} \to \mathbb{C}$ (where $\mathbb{C}^*$ acts on $\mathbb{C}$ by multiplication)
and a $\mathbb{C}^*$-equivariant ample line bundle $\mathcal{L}$ over $\mathcal{X}$. 
In addition we require the fibers $(\mathcal{X}_t,  \mathcal{L}|_{\mathcal{X}_t})$ over $t \neq 0$ is isomorphic to $(X,L)$. 
\end{definition}


Note that any test configuration is equivariantly embedded into $\mathbb{C}P^N \times \mathbb{C}$.
Here the $\mathbb{C}^*$-action on $\mathbb{C}P^N$ is given by a one parameter subgroup of $SL(N+1,\mathbb{C})$.
If  $D\subset X$ is any irreducible divisor, the one parameter subgroup of $SL(N+1,\mathbb{C})$ associated to any test configuration of $(X,L)$ 
also induces a test configuration $(\mathcal{D},\mathcal{L}|_{\mathcal{D}})$ of $(D,L|_{D})$.

To define the log Futaki invariant, fix a test configuration $(\mathcal{X}, \mathcal{L})$ for $(X,L)$ and a irreducible divisor $D\subset X$.
Let $d_k$ and $\tilde{d}_k$ be dimensions of $H^0(\mathcal{X}_0, \mathcal{L}^k|_{\mathcal{X}_0})$ 
and $H^0(\mathcal{D}_0, \mathcal{L}^k|_{\mathcal{D}_0})$ respectively and
let $w_k$ and $\tilde{w}_k$ be total weights of $\mathbb{C}^*$ actions on these cohomologies respectively.
For $k\gg 0$, we then have asymptotic expansions (set $n = \dim X$)
\begin{eqnarray*}
d_k &=& a_0k^n + a_1k^{n-1} + \mathcal{O}(k^{n-2}), \quad \tilde{d}_k= \tilde{a}_0k^{n-1} + \mathcal{O}(k^{n-2}),\\
w_k &=& b_0k^{n+1} + b_1k^n + \mathcal{O}(k^{n-1}) \quad\text{and}\quad  \tilde{w}_k = \tilde{b}_0k^n + \mathcal{O}(k^{n-1}).
\end{eqnarray*}

\begin{definition}
The {\it log Futaki invariant} for $(\mathcal{X}, \mathcal{D}, \mathcal{L})$ and $\beta \in (0,1]$ is defined as
\[
\mathrm{Fut}(\mathcal{X}, \mathcal{D}, \mathcal{L}, \beta) = 2\Bigl(\frac{a_1}{a_0}b_0-b_1\Bigr) + (1-\beta)\Bigl(\tilde{b}_0-\frac{\tilde{a}_0}{a_0}b_0\Bigr).
\]
\end{definition}

\begin{definition}
$(X,D,L,\beta)$ is {\it log K-semistable} if 
$
\mathrm{Fut}(\mathcal{X}, \mathcal{D}, \mathcal{L}, \beta) \leq 0
$
for every test configuration $(\mathcal{X},  \mathcal{L})$ for $(X,L)$.
\end{definition}

For every toric variety $(X_P, L_P)$, Donaldson \cite{D2} showed that every rational convex piecewise linear function
$h : P \to \mathbb{R}$ induced a test configuration for $(X_P, L_P)$ with 
\begin{eqnarray*}
d_k &=&  k^n\mathrm{Vol}(P) + \frac{k^{n-1}}{2}\mathrm{Vol}(\partial P) + \mathcal{O}(k^{n-2}),\\
w_k &=& k^{n+1}\int_P (R-h)d\nu  + \frac{k^n}{2} \int_{\partial P} (R-h)d\sigma + \mathcal{O}(k^{n-1})
\end{eqnarray*}
where $R$ is an integer such that $h\leq R$.
For any toric divisor $D_F \subset X_P$ defined by a facet $F$ of $P$,  
every such function $h$ induces a test configuration for $(D_F, L_P)$ with
\begin{eqnarray*}
\tilde{d}_k &=& k^{n-1}\mathrm{Vol}(F) + \mathcal{O}(k^{n-2}),\\
\tilde{w}_k &=& k^n \int_F (R-h)d\sigma + \mathcal{O}(k^{n-1}).
\end{eqnarray*}

As analogy with Donaldson \cite{D2}, we define the log K-semistability for toric degenerations.
\begin{definition}
$(X_P, D_F, L_P, \beta)$ is {\it log K-semistable for toric degenerations} 
if for every rational convex piecewise linear function $h : P \to \mathbb{R}$, the induced log Futaki invariant is nonpositive:
\[
\frac{\mathrm{Vol}(\partial P)}{\mathrm{Vol}(P)}\int_P h d\nu -\int_{\partial P} h d\sigma 
+(1-\beta) \Bigl( \int_F h d\sigma - \frac{\mathrm{Vol}(F)}{\mathrm{Vol}(P)} \int_P h d\nu \Bigr) \leq 0.
\]
\end{definition}

Then we show the following as a application of Theorem \ref{main}.
\begin{proposition}\label{K-semistability}
If the Chow form of $(X_P, D_F, L_P^i, \beta)$
is $T_{iP}$-semistable for all $i\gg 0$ then this pair  is K-semistable for toric degenerations.
\end{proposition}

\begin{remark}
In particular, the asymptotic log Chow semistability for $(X_P, D_F, L_P^i, \beta)$ implies the log K-semistability for toric degenerations by Proposition \ref{Hilbert-Mumford}.
\end{remark}

\begin{proof}
We first fix $h : P \to \mathbb{R}$ as a rational convex piecewise linear function.
Then there is a positive integer $k$ such that $g:=-h \in \mathrm{PL}_{\mathbb{Q}}(P,k)$ (Remark \ref{Q}).
By Theorem \ref{main}, 
for all $i\gg 0$,
$T_{iP}$-semistability for  the Chow form of $(X_P, D_F, L_P^i, \beta)$ implies that
\begin{equation}\label{weight}
E_P(ik)  \biggl( 2ik \int_P g d\nu + (1-\beta)\int_F g d\sigma \biggr)
- \biggl( 2ik \mathrm{Vol}(P) + (1-\beta)\mathrm{Vol}(F) \biggr) \sum_{\mathbf{a} \in P \cap (\mathbb{Z}/ik)^n} g(\mathbf{a}) \geq 0
\end{equation}
holds for all $i\gg 0$.
By Lemma 3.3 in \cite{ZZ} we note that
\begin{eqnarray*}
E_P(ik) &=&  (ik)^n\mathrm{Vol}(P) + \frac{(ik)^{n-1}}{2}\mathrm{Vol}(\partial P) + \mathcal{O}((ik)^{n-2}),\\
\sum_{\mathbf{a} \in P \cap (\mathbb{Z}/ik)^n} g(\mathbf{a}) &=& (ik)^{n}\int_P gd\nu  + \frac{(ik)^{n-1}}{2} \int_{\partial P} gd\sigma + \mathcal{O}((ik)^{n-1}).
\end{eqnarray*}
Therefore the left hand side of (\ref{weight}) is equal to
\[
\frac{(ik)^n \mathrm{Vol}(P)}{2}
\biggl\{
\frac{\mathrm{Vol}(\partial P)}{\mathrm{Vol}(P)}\int_P g d\nu -\int_{\partial P} g d\sigma 
+(1-\beta) \Bigl( \int_F g d\sigma - \frac{\mathrm{Vol}(F)}{\mathrm{Vol}(P)} \int_P g d\nu \Bigr)
\biggr\} 
+ \mathcal{O}((ik)^{n-1}).
\]
This complete the proof.
\end{proof}
\section{Relation to the existence of conical K\"{a}hler Einstein metrics}
By recent progresses in Donaldson-Tian-Yau conjecture, 
it is natural to ask the equivalence between the existence of conical K\"{a}hler Einstein metrics and the log Chow stability.
In this section we check the obstruction as in Corollary \ref{cor} 
for a few toric Fano manifolds admitting conical K\"{a}hler Einstein metrics.
See \cite{L} for the definition of conical metrics and more details.

Let $X_P$ be the toric Fano manifold associated from a reflexive polytope $P$ 
(that is a Delzant polytope satisfying $\mathrm{Int}(P)\cap\mathbb{Z}^n = \{0\}$),
and let $D=\sum_t (1-\beta_t)D_{F_t}$ be a formal sum of different toric divisors associated from different facets $F_t$ of $P$ where $\beta_t \in (0,1].$
Let us consider
\[
Q_i(X_P, D) :=
E_P(i)  \biggl( 2i \int_P \mathbf{x} d\nu + \sum_t(1-\beta_t)\int_{F_t} \mathbf{x} d\sigma \biggr)
- \biggl( 2i \mathrm{Vol}(P) + \sum_t (1-\beta_t)\mathrm{Vol}(F_t) \biggr) \sum_{\mathbf{a} \in P \cap (\mathbb{Z}/i)^n} \mathbf{a}.
\]
By Corollary \ref{cor} and Remark \ref{extension}, if the pair $(X_P,  D)$ is asymptotically log Chow semistable then $Q_i$ must vanish for every integer $i$.

\begin{example}
Let $P$ be an interval $[-1,1]$.
The corresponding toric Fano manifold $X_P$ is $\mathbb{C}P^1$.
Let $D=(1-\beta_0)[0]+(1-\beta_{\infty})[\infty]$ be a formal sum of toric divisors
where $[0]$ and $[\infty]$ are divisors of $X_P$ corresponding $-1, 1 \in P$ respectively.
By a simple calculation, we have
\[
Q_i(X_P, D)=\frac{1}{2}(i+1)(\beta_0-\beta_{\infty}).
\]
Thus if this pair $(X_P,D)$ is asymptotic log Chow semistable then $\beta_0=\beta_{\infty}$ must be hold.
This condition $\beta_0=\beta_{\infty}$ is equivalent to the existence of a conical K\"{a}hler Einstein metric whose cone angles on $[0]$ and $[\infty]$ are $\beta_0$ .
\end{example}

\begin{example}
Let $P$ be a polygon defined as
\[
\Set{(x,y) \in \mathbb{R}^2 | -x-y+1\geq 0, \quad x+1\geq 0,\quad x+y+1\geq 0,\quad y+1\geq 0}.
\]
The corresponding toric Fano manifold $X_P$ is the first Hilzebruch surface 
$\mathbb{P}(\mathcal{O}_{\mathbb{C}P^1}\oplus \mathcal{O}_{\mathbb{C}P^1}(-1)).$
Let $D_1, D_2$ and $D_{\infty}$ be toric divisors corresponding to $P\cap\Set{x+1=0}, P\cap\Set{y+1=0}$ and $P\cap\Set{-x-y+1=0}$ respectively. 
Note that $D_1, D_2$ and $D_{\infty}$ are divisors of fibers over $0,\infty\in\mathbb{C}P^1$ and the $\infty$-section respectively.
In a part of \cite[Theorem 3.14]{SW}, Song and Wang proved the existence of a conical K\"{a}hler Einstein metric 
whose cone angles are 13/14, 13/14 and 5/7 along $D_1, D_2$ and $D_{\infty}$ respectively.
Set \[D=(1-\frac{13}{14})D_1+(1-\frac{13}{14})D_2+(1-\frac{5}{7})D_{\infty}.\]
By a simple calculation we have
\[
Q_i(X_P,D)=(\frac{21}{10}i-\frac{1}{7})(1,1)^T.
\]
Therefore this pair $(X_P,D)$ is asymptotically Chow unstable.
\end{example}

\bigskip

\address{
Mathematical Institute \\ 
Tohoku University \\
Sendai 980-8578 \\
Japan
}
{sb3m25@math.tohoku.ac.jp}

\end{document}